\documentclass[a4paper,12pt]{article}
\usepackage{amssymb, amsmath, afterpage}
\usepackage{amsfonts}
\usepackage[dvips]{graphics}
\newcommand{\mybox}{\hfill $\square$}
\newcommand{\+}[1]{\ensuremath{\overset{+}{#1}}}
\newcommand{\pli}[1]{\ensuremath{\overset{+}{\imath}}}
\newcommand{\mi}[1]{\ensuremath{\overset{-}{#1}}}

\newcommand{\sym}{\ensuremath{\mathrm{Sym}}}
\newcommand{\dih}{\ensuremath{\mathrm{Dih}}}
\newcommand{\I}{\ensuremath{\mathcal{I}}}
\newcommand{\J}{\ensuremath{\mathcal{J}}}
\newcommand{\ip}[2]{\ensuremath{\langle{#1},{#2}\rangle}} 

\newcommand{\supp}{\ensuremath{\mathrm{supp}}}
\newcommand{\qed}{\hfill $\square$ \\}

\newtheorem{thm}{Theorem}[section]
\newtheorem{lemma}[thm]{Lemma}         
\newtheorem{prop}[thm]{Proposition}

\newenvironment{proof}{\paragraph{Proof}}{\mybox}
\begin{document}
    \title{\Large\textbf{Zero Excess and Minimal Length in Finite Coxeter Groups}}
\date{}
      \author{S.B. Hart and P.J. Rowley\thanks{The authors wish to acknowledge partial support for this work from
      the Birkbeck Faculty of Social
      Science and the Manchester Institute for Mathematical Sciences (MIMS).
      }}
  \maketitle
\vspace*{-10mm}
\begin{abstract} \noindent Let $\mathcal{W}$ be the set of strongly real elements of $W$, a Coxeter group.
Then for $w \in \mathcal{W}$, $e(w)$, the excess of $w$, is
defined by $e(w) = \min\{\ell(x) + \ell(y) - \ell(w) \; | \; w=xy,
x^2 = y^2 = 1\}$. When $W$ is finite we may also define $E(w)$,
the reflection excess of $w$. The main result established here is
that if $W$ is finite and $X$ is a $W$-conjugacy class, then there
exists $w \in X$ such that $w$ has minimal length in $X$ and $e(w)
= 0 = E(w)$. (MSC2000: 20F55)
\end{abstract}

\section{Introduction}\label{intro}

Suppose that $W$ is a Coxeter group and let $\mathcal{W}$ denote
the set of strongly real elements of $W$. So

$$ \mathcal{W} =
\{ w \in W \;| \; w = xy \; \mbox{where} \; x, y \in W \; \mbox{and} \; x^2 = y^2 = 1 \}.$$
Let $w \in \mathcal{W}$. Then the \textit{excess} of $w$, $e(w)$, is defined by

$$e(w) = \min\{\ell(x) + \ell(y) - \ell(w) \; | \; w=xy, x^2 = y^2 =
1\}.$$
In \cite{invprod} it is shown that any element $w \in \mathcal{W}$ is $W$-conjugate
to an element whose excess equals zero. Or, in other words, $w$ is $W$-conjugate to an
element $xy$ where $x$ and $y$ are involutions or the identity element and
$\ell(xy) = \ell(x) + \ell(y)$. The present paper explores this theme further
in the case when $W$ is finite. When $W$ is
finite it is well-known that $W = \mathcal{W}$. From Carter's seminal paper \cite{carter}, it
follows that, in fact, each $w \in W$ can be expressed in the form
$w = xy$ where $x^2 = y^2 = 1$ and
$$V_{-1}(x) \cap V_{-1}(y) = \{0\}.$$ Here $V$ is a reflection module for $W$
(to be defined in Section \ref{back}) and  $V_{\lambda}(x)$
denotes the $\lambda$-eigenspace of $x$ on $V$. We note
\cite{carter} only considers the Weyl groups -- for a proof covering all finite Coxeter groups
see Lemma \ref{h3h4}. Writing $L(w)$ for the reflection length of
$w$, this means that
$$L(w) = L(x) + L(y).$$

With this in mind we define the {\sl reflection excess} of $w$,
denoted $E(w)$, to be

$$E(w) = \min\{\ell(x) + \ell(y) - \ell(w) \; | \; w=xy, x^2 = y^2 =
1, L(w) = L(x) + L(y)\}.$$

\begin{table}[h!tb]\begin{center}\begin{tabular}{cccr}
$x$&$y$&$L(x)+L(y)$&$\ell(x)+\ell(y)$\\
\hline %
$(14)(23)$&$(15)(26)$&4&$6+12=18$\\
$(14)(36)$&$(15)(23)$&4&$8+8=16$\\
$(14)(26)$&$(15)(36)$&4&$10+10=20$\\
$(15)(23)$&$(45)(26)$&4&$8+8=16$\\
$(15)(36)$&$(45)(23)$&4&$10+2=12$\\
$(15)(26)$&$(45)(36)$&4&$12+6=18$\\
$(45)(23)$&$(14)(26)$&4&$2+10=12$\\
$(45)(36)$&$(14)(23)$&4&$6+6=12$\\
$(45)(26)$&$(14)(36)$&4&$8+8=16$\\
$(12)(46)(35)$&$(13)(24)(56)$&6&$5+5=10$\\
$(13)(24)(56)$&$(16)(25)(34)$&6&$5+15=20$\\
$(16)(25)(34)$&$(12)(35)(46)$&6&$15+5=20$%
\end{tabular}\end{center}\caption{\label{145236}$(145)(236) = xy$}\end{table} %

It is clear that $E(w) \geq e(w)$. However, it is not always the case that $E(w) = e(w)$.
Consider the element $w=(145)(236)$ of $\sym(6)$, the Coxeter group of type $A_5$. We have $L(w) = 4$,
$\ell(w) = 10$ and Table \ref{145236} gives the possibilities for $x$ and $y$ with $w=xy$. As can be
seen, $E(w) = 2$ but $e(w) = 0$. Surprisingly, the difference between $E(w)$ and $e(w)$
can be arbitrarily large, as is
shown in Proposition \ref{bigXS}.\\

The main theorem in this paper is

\begin{thm} \label{ex0minlength} Suppose that $W$ is a finite Coxeter group. If $X$
is a conjugacy class of $W$, then there exists $w \in X$ of minimal length in $X$ such that
 $e(w) = E(w) = 0$.\end{thm}

We do not know if it is true that for an arbitrary Coxeter group and a strongly real conjugacy class
$X$, there exists $w \in X$ with $w$ of minimal length in $X$ and $e(w) = 0$. However we note
that this does hold for an arbitrary Coxeter group when $X$ is a
strongly real conjugacy class whose elements have finite order.
This may be seen by combining a theorem of Tits (Chap V, Section
4, Ex 2d of \cite{titsref}) with Lemma \ref{minpara} (which holds
in general)
and Theorem \ref{ex0minlength}.\\

This paper is arranged as follows. Our next section gathers together
relevant background material while reviewing much of the standard
notation used for Coxeter groups. Section \ref{minzero}, apart from proving Proposition \ref{bigXS},
focusses on the proof of Theorem \ref{ex0minlength}. Part of the proof
involves checking, with the aid of {\sc
Magma} \cite{magma}, all the cuspidal classes of the exceptional
finite irreducible Coxeter groups. The data resulting from these
calculations is documented in the appendix. At present, when
studying minimal elements in conjugacy classes, this case-by-case
approach is often the best we can do -- see for example Chapter 3 of
Geck and Pfeiffer \cite{gkpf}.\\

\section{Preliminary Results and Notation}\label{back}

From now on we assume that $W$ is a finite Coxeter group and quickly review
standard notation and facts about such groups. So $W$ has a
presentation of the form

\[ W = \langle R \; | \; (rs)^{m_{rs}} = 1, r, s \in R\rangle \]
where $m_{rs} = m_{sr} \in \mathbb{N}$, $m_{rr} = 1$ and $m_{rs}
\geq 2$ for $r, s \in R, r\neq s$.  The rank of $W$ is $|R|$. The length of an element $w$
of $W$, denoted by $\ell(w)$, is defined to be
\[\ell(w) = \left\{ \begin{array}{l} \min\{ l \; | \; w=r_{1}\cdots r_{l}, r_{i} \in R\} {\mbox{ if $w \neq 1$}} \\
0 {\mbox{ if $w=1$}.} \end{array}\right.\]

Let $V$  be a real euclidean vector space with basis $\Pi = \{ \alpha_r \;
| \; r \in R\}$ and norm  $|| \; ||$. For $r,s \in R$

\[\ip{\alpha_r}{\alpha_s} =  -||\alpha_r|| \; ||\alpha_s|| \; \cos\left(\frac{\pi}{m_{rs}}\right)
\]

defines a symmetric bilinear form $\ip{ \; }{ \; }$
on $V$. Letting $r, s \in R$ we define
 $$r\cdot\alpha_s = \alpha_s - \frac{2\ip{\alpha_r}{\alpha_s}}{\langle\alpha_r, \alpha_r\rangle}\alpha_r.$$
This then extends to an action of $W$ on $V$ which is both
faithful and respects the bilinear form $\ip{ \; }{ \; }$ (see
\cite{humphreys}). The elements of $R$ act as reflections upon $V$
and $V$ is referred to as a reflection module for $W$. The subset
$\Phi = \{w\cdot\alpha_r \; | \; r \in R, w \in W\}$ of $V$ is the
root system of $W$, and $\Phi^+ = \{\sum_{r\in R}\lambda_r\alpha_r
\in \Phi \; | \; \lambda_r \geq 0 {\mbox{ for all $r$}}\}$ and
$\Phi^- = -\Phi^+$ are, respectively, the positive and negative
roots of $\Phi$. For $w \in W$, let $N(w) = \{ \alpha \in \Phi^+
\; | \; w\cdot\alpha \in \Phi^- \}$ -- it is an important fact
that $\ell(w) = |N(w)|$. In a similar vein we have the following well-known result.

\begin{lemma} \label{lengthadd} Let $g, h \in W$. Then
$$N(gh) = N(h)\setminus [-h^{-1}N(g)] \cup h^{-1}[N(g)\setminus
N(h^{-1})].$$ %
Hence $\ell(gh) = \ell(g) + \ell(h) -2\mid N(g)\cap
N(h^{-1})\mid$.\end{lemma}

\begin{proof} See, for example, Lemma 2.2 in \cite{invprod}.
\end{proof}\\

For $J$ a subset of $R$ define $W_J$ to be the subgroup generated
by $J$. Such a subgroup of $W$ is referred to as a standard
parabolic subgroup. Standard parabolic subgroups are Coxeter
groups in their own right with root system
\[\Phi_J = \{w\cdot\alpha_{r} \; | \; r \in J, w \in W_J\}\]
(see Section 5.5 of \cite{humphreys} for more on this). A
conjugate of a standard parabolic subgroup is called a parabolic
subgroup of $W$. Finally, a {\sl cuspidal} element of $W$ is an
element which is not contained in any proper parabolic subgroup of
$W$. Equivalently, an element is cuspidal if its $W$-conjugacy
class has empty intersection with all the proper standard
parabolic subgroups of $W$.

\begin{thm} \label{tits}  Let $0 \neq v \in V$.
Then the stabilizer of $v$ in $W$ is a proper parabolic subgroup of
$W$.
\end{thm}

\begin{proof}Consult Ch V \S 3 Proposition 2 of \cite{titsref}.
\end{proof}

\begin{lemma}\label{minpara} Suppose that $J \subseteq R$ and $w \in W_J$. Then

$$\min \{\ell(h^{-1}wh) \; | \; h \in W_J \} = \min \{\ell(g^{-1}wg) \; | \; g \in
W \}.$$
\end{lemma}

\begin{proof} See Lemma 3.1.14 of \cite{gkpf}.
\end{proof}

\begin{lemma} \label{h3h4} Suppose $w \in W$.
Then there exist $x, y \in W$ with $w=xy$, $x^2 = y^2 = 1$
and $V_{-1}(x) \cap V_{-1}(y)=
\{0\}.$
\end{lemma}

\begin{proof} The proof is by induction on the rank of $W$, the
rank 1 case being trivial. Let $x, y \in W$ be such that $w = xy$
with $x^2 = y^2 = 1$. (For $W$ a Weyl group, this is possible by
\cite{carter}. The case when $W$ is a dihedral group is
straightforward to verify while types $H_3$ or $H_4$ may be
checked using {\sc
Magma}  \cite{magma}.) If $V_{-1}(x) \cap V_{-1}(y) =
\{0\}$, we are done. So suppose $0 \ne v \in V_{-1}(x) \cap
V_{-1}(y)$. By Theorem \ref{tits}, $w$ is contained in a proper
parabolic subgroup of $W$. Hence $w$ is conjugate to an element
$u$ of some proper standard parabolic subgroup $W_J$ of $W$. By
induction $u =ab$ for some $a, b \in W_J$ where $a^2 = b^2 = 1$
and $V_{-1}(a) \cap V_{-1}(b) = \{0\}$. The appropriate conjugates
of $a$ and $b$ will have the same properties with respect to $w$,
so proving the lemma.
\end{proof}

\begin{lemma}\label{coxelts} Suppose that $R = \{r_1, r_2, \ldots, r_n\}$ and set $w = r_1r_2 \cdots r_n$. Let $X$ denote the
$W$-conjugacy class of $w$. Then
\begin{trivlist}
\item (i) the minimal length in $X$ is $n$; and
\item (ii) the product of $r_1, \dots ,r_n$ in any order is an element
of $X$.
\end{trivlist}
\end{lemma}

\begin{proof} See Proposition 3.1.6 of \cite{gkpf}.
\end{proof}\\

Note that the minimal length elements of $X$ in Lemma \ref{coxelts} are known as
the {\em Coxeter elements} of $W$. We shall use $\dih(2m)$ to denote the dihedral group of order $2m$. The famous classification of irreducible finite Coxeter groups
obtained by Coxeter \cite{coxeter} (see also \cite{humphreys})
states

\begin{thm}\label{classification} An irreducible finite Coxeter group is either of type
$A_n (n\geq 1)$, $B_n (n \geq 2)$, $D_n (n \geq 4)$, $\dih(2m) (m \geq 5)$, $E_6$, $E_7$, $E_8$, $F_4$, $H_3$
or $H_4$.
\end{thm}

We next discuss concrete descriptions of the Coxeter groups of types $A_n, B_n$ and $D_n$
which will feature in a number of our proofs. We
draw the reader's attention to the fact that when regarding $W$ as a group of
permutations or signed permutations we act on the right, as is customary for
permutations. First, $W(A_n)$ may be viewed as being
$\sym(n+1)$ with the set of fundamental reflections given by $\{(12), (23), \dots , (n \;
n+1) \}$. The elements of $W(B_n)$ can be thought of as signed permutations of $\sym(n)$.
We say a cycle in an element of $W(B_n)$ is of negative sign type if it has an odd number
of minus signs, and positive sign type otherwise. The set of fundamental reflections in
$W(B_n)$ can be taken to be $\{(\overset{+}{1}\overset{+}{2}),
(\overset{+}{2}\overset{+}{3}), \ldots, (\overset{+}{n-1} \; \overset{+}{n}),
(\overset{-}{n})\}$. An element $w$ expressed as a product $g_1g_2\cdots g_k$ of disjoint
signed cycles is {\em positive} if the product of all the sign types of the cycles is
positive, and negative otherwise. The group $W(D_n)$ consists of all positive elements of
$W(B_n)$. The fundamental reflections of $W(D_n)$ can be taken to be the set
$\{(\overset{+}{1}\overset{+}{2}), (\overset{+}{2}\overset{+}{3}), \ldots,
(\overset{+}{n-1} \; \overset{+}{n}), (\overset{-}{n-1}\; \overset{-}{n})\}$. \\ 

Let $\{e_i\}$ be an orthonormal basis for the reflection module $V$. For $\sigma \in W(A_n)$, define $e_i\cdot \sigma = e_{i\sigma}$. For $1 \leq i \leq n$ we may set $\alpha_{(i \; i+1)}$ to be $e_i - e_{i+1}$. The positive roots of $W(A_n)$ are then $e_i - e_j$ for $1 \leq i < j \leq n$.
The positive roots of $B_n$ are of the form $e_i \pm e_j$ for $1 \leq i < j \leq n$ and $e_i$
for $1\leq i \leq n$ and, for example, $(e_1 + e_2)\cdot (\mi 1 \+ 2 \mi 4) = -e_2 + e_4$.
The positive roots of $D_n$ are of the form $e_i \pm e_j$ for $1
\leq i < j \leq n$. Therefore the root system of $D_n$ consists of the long roots of the
root system for $B_n$. Even if $w$ is positive, it may contain negative cycles, which we
wish on occasion to consider separately, so when considering elements of $W(D_n)$ we
often work in the environment of $W(B_n)$ to avoid ending up with non-group elements.  \\

The following result is obtained from Propositions 3.4.6, 3.4.7,
3.4.11, 3.4.12 of \cite{gkpf}. The length of minimal length elements
in conjugacy classes is not given explicitly, but expressions for
elements of minimal length are given in Section 3.4.2 and from these
the length can be easily calculated.

\begin{prop} \label{conjbd} Let $W$ be of type $B_n$ or $D_n$. Then the following
hold.
\begin{trivlist}%
\item{(i)} Conjugacy classes in $W$ are parameterized by signed
cycle type, with one class for each signed cycle type except in the case where all cycles
are of even length and positive, and $W$ is of type $D_n$. In this case there are two
conjugacy classes, which can be
interchanged by the length-preserving graph automorphism.%
\item{(ii)} Cuspidal conjugacy classes in $W$ are those whose
signed cycle type consists only of negative cycles. %
\item{(iii)} Each cuspidal conjugacy class $X$ corresponds to a
non-increasing partition $$\lambda = (\lambda_1, \lambda_2,
\ldots, \lambda_k)$$ of $n$, where each $\lambda_i$ is the length
of a negative cycle of $w \in X$. Let $\mu_i = \sum_{j = i+1}^k
\lambda_i$. The minimal length of an element of $X$ is%
$$ \left\{\begin{array}{ll} \sum_{i=1}^k \left( \lambda_i +
2\mu_i\right) & {\mbox{ if $W = W(B_n)$}}\\
\sum_{i=1}^k \left( \lambda_i -1 + 2\mu_i\right) & {\mbox{ if $W =
W(D_n)$}}\end{array}. \right.$$%
\end{trivlist}\end{prop}

We close this section with a few pieces of notation. For a subset
$X$ of $W$ we define $\ell_{\min}(X) = \min \{\ell(x) \;| \; x \in
X \}$. The longest element of $W$ will be denoted by $w_0$. When
$W = W(A_{n-1}) \cong \sym(n)$, $W$ acts in the usual way on the
set $\Omega = \{1, 2, \ldots, n\}$. In this case, for $w \in W$,
the {\em support} of $w$, $\mathrm{supp}(w)$, is given by
$\mathrm{supp}(w) = \{ \delta \in \Omega \;|\; \delta w \neq \delta \}$.

\section{Minimal Length and Zero Excess and Reflection Excess } \label{minzero}

We begin with an elementary lemma.

\begin{lemma}\label{basics} For $w \in W$, the following
hold.
\begin{trivlist}
\item{(i)} Both $e(w)$ and $E(w)$ are non-negative and even.%
 \item{(ii)}
If $w$ is an involution or the identity element, then $e(w) = E(w) = 0$.%
\item{(iii)} $\ell(w)$ is the sum of the lengths, $e(w)$ is the sum
of the excesses, and $E(w)$ is the sum of the reflection excesses
of the projections of $w$ into the irreducible direct factors of $W$.\end{trivlist}\end{lemma}

\begin{proof} For (i), suppose $x^2 = y^2 = 1$
and $xy = w$. Then, using Lemma \ref{lengthadd}, $\ell(w) = \ell(x)
+ \ell(y) - 2|N(x) \cap N(y)|$ and hence $\ell(x) + \ell(y) -
\ell(w)$ is even and (i) follows. The remaining parts of the lemma are
straightforward.\end{proof}\\

In the next lemma we encounter the following two subsets of $W$:-
$$\I_w = \{x \in W \; | \; x^2 = 1, w^{x} = w^{-1}\} \; \mbox{and}$$

$$\J_w = \{x \in W \; | \; x^2 = 1, w^{x} = w^{-1}, V_1(w) \subseteq
V_1(x)\}.$$

\begin{lemma}\label{Jlemma}Suppose that $w \in W$.
\begin{trivlist}
\item{(i)} $\J_w$ is the set of $x$ such that $w = xy$ where $x^2 =
y^2 =1$ and $L(w) = L(x) + L(y)$.
\item{(ii)} If $w$ is cuspidal, then $\I_w = \J_w$ and, in
particular, $e(w) = E(w)$.
\end{trivlist}
\end{lemma}

\begin{proof}
Suppose $x \in \J_w$. Set $y = xw$. Clearly $y^2 = 1$. Let $v \in
V_{-1}(x) \cap V_{-1}(y)$. Then $y\cdot v = x\cdot v$ and hence
$w\cdot v = v$. That is, $v \in V_1(w)$. But $V_1(w) \subseteq
V_1(x)$ as $x \in \J_w$. Thus $v \in V_1(x) \cap V_{-1}(x)$ and so
$v=0$. Therefore $V_{-1}(x)
\cap V_{-1}(y) = \{0\}$, whence $L(w) = L(x) + L(y)$.\\
On the other hand, suppose $x^2 = y^2 = 1$, $w=xy$ and $L(w) = L(x)
+ L(y)$. Then $V_{-1}(x) \cap V_{-1}(y) = \{0\}$. Clearly $w^{x} =
w^{-1}$. Suppose $v \in V_1(w)$. Then $v = w\cdot v = xy\cdot v$.
Hence $x\cdot v = y \cdot v$. Therefore $x\cdot v - v = y\cdot v -
v$. But $x\cdot v - v \in V_{-1}(x)$ and $y\cdot v - v \in
V_{-1}(y)$, which forces $x\cdot v - v = y\cdot v - v = 0$. That is,
$v \in V_1(x)$. Therefore $V_1(w) \subseteq V_1(x)$ and $x \in
\J_w$, which establishes (i).\\

For (ii), $w$ being cuspidal implies, by Theorem \ref{tits}, that
$V_1(w) = 0$. So $\I_w = \J_w$, and clearly $e(w) = E(w)$.

\end{proof}

We now give examples of elements in the symmetric group which have
arbitrarily large reflection excess, while having zero excess,
showing that these quantities can differ by arbitrarily large
amounts.

\begin{prop}\label{bigXS} Let $W = W(A_{n - 1}) \cong \sym(n)$ and assume $0 < 4k \leq
n$. Define \begin{eqnarray*} w_1 &=& (1\; 4\; 6\; 8\cdots 4k-4 \; 4k-2
\;\;\; 4k-1),\\ w_2 &=& (2\;\; 4k\;\;\; 4k-3 \; 4k-5 \cdots 7\; 5\;
3)\end{eqnarray*} and $w = w_1w_2$. Then $e(w) = 0$ but $E(w) \geq
4(k-1)^2$.\end{prop}
\begin{proof} First we observe that $\supp(w_1)\cap \supp(w_2) = \emptyset$. We define two involutions $x, y$ as follows.
\begin{eqnarray*} x&=& (13)(2 \; 4k-1)\, (45)(67)\cdots (4k-6 \; 4k-5)(4k-4 \;
4k-3)\, (4k-2 \; 4k)\\
y&=& (12)(34)\cdots (4k-1 \; 4k)\end{eqnarray*}
A simple check shows $w = xy$. Now $N(y) = \{e_{2i-1} - e_{2i} | 1
\leq i \leq 2k\}$. To find $N(x) \cap N(y)$, note that
\begin{eqnarray*}
(e_{1}-e_{2})\cdot x &=& e_{3} - e_{4k-1} \in \Phi^+\\
(e_{3}-e_{4})\cdot x &=& e_{1} - e_{5} \in \Phi^+\\
(e_{4k-3}-e_{4k-2})\cdot x &=& e_{4k-4} - e_{4k} \in \Phi^+\\
(e_{4k-1}-e_{4k})\cdot x &=& e_{2} - e_{4k-2} \in \Phi^+\\
(e_{2i-1}-e_{2i})\cdot x &=& e_{2i-2} - e_{2i+1} \in \Phi^+ \; (3\leq
i \leq 2k-2).
\end{eqnarray*}
Hence $N(x) \cap N(y) = \emptyset$, which means $e(w) = 0$ by Lemma \ref{lengthadd}.\\
Next we consider $E(w)$. Since $w_1 = (1\; 4\; 6\; 8\cdots 4k-2
\;\;\; 4k-1)$ and $1 < 4 < \cdots < 4k-1$, we can compare $w_1$
seen as an element of $\sym(\supp(w_1))$ with the element $(1\;
2\cdots 2k)$ as an element of $\sym(2k)$. This allows us to use
[Proposition 2.7(iv), \cite{invprod}] and Lemma \ref{Jlemma}(ii)
to deduce that $E(w_1) \geq E((12 \cdots 2k)) = 2(k-1)^2$. This is
because whenever $w_1 = \sigma_1\tau_1$ where $\sigma_1, \tau_1$
are involutions in $\sym(\supp(w_1))$, we have that
$$\left(N(\sigma_1) \cap N(\tau_1)\right)|_{\sym(\supp(w_1))} \subseteq
N(\sigma_1)\cap N(\tau_1).$$ %
Similarly, by comparing $w_2 = (2\;\; 4k\;\;\; 4k-3 \cdots 7\; 5\;
3)$ to $(1\;\; 2k\;\; 2k-1 \cdots 3\; 2) = (1\; 2\cdots 2k)^{-1}$ we
deduce that $E(w_2)
\geq 2(k-1)^2$.\\

Suppose $w = \sigma\tau$, where $\sigma$ and $\tau$ are involutions with $L(\sigma) + L(\tau) = L(w)$. Then $w^{-1} = w^{\sigma} = w_1^{\sigma}w_2^{\sigma}$. So, as $w^{-1} = w_1^{-1}w_2^{-1}$, we have $$w_1^{\sigma}w_2^{\sigma} = w_1^{-1}w_2^{-1}.$$ Since $\supp(w_1)\cap \supp(w_2) = \emptyset$, it follows that either $w_1^{\sigma} = w_1^{-1}$ or $w_1^{\sigma} = w_2^{-1}$. Suppose for the moment that $w_1^{\sigma} = w_2^{-1}$. Then $w_1^{\tau} = w_2^{-1}$. In particular, both $\sigma$ and $\tau$ interchange $\supp(w_1)$ and $\supp(w_2)$. Put $$v = \sum_{i \in \supp(w_1)} e_i - \sum_{j \in \supp(w_2)} e_j.$$
Then $0 \neq v \in V_{-1}(\sigma) \cap V_{-1}(\tau)$, which contradicts the fact that $L(w) = L(\sigma) + L(\tau)$. Therefore we have $w_1^{\sigma} = w_1^{-1}$, and so $w_2^{\sigma} = w_2^{-1}$. Similarly $w_1^{\tau} = w_1^{-1}$ and $w_2^{\tau} = w_2^{-1}$. As a consequence every 2-cycle of $\sigma$ and every 2-cycle of $\tau$ has its support either in $\supp(w_1)$ or $\supp(w_2)$. Therefore there are involutions $\sigma_1, \sigma_2, \tau_1$ and $\tau_2$ such that $\sigma = \sigma_1\sigma_2$, $\tau = \tau_1\tau_2$ where $w_1 =
\sigma_1\tau_1$, $w_2 = \sigma_2\tau_2$, $\supp(\sigma_1) \cup
\supp(\tau_1) \subseteq \supp(w_1)$ and $\supp(\sigma_2) \cup
\supp(\tau_2) \subseteq \supp(w_2)$. Any $\{i, j\} \subseteq
\supp(w_1)$ (and hence $iw_1, jw_1$) will thus be fixed by
$\sigma_2$ and $\tau_2$. So if $e_i - e_j \in N(\sigma_1)\cap
N(\tau_1)$, then $e_i - e_j \in N(\sigma)\cap
N(\tau)$. Hence\\
$$\left(N(\sigma_1) \cap N(\tau_1)\right) \subseteq
N(\sigma)\cap N(\tau)|_{\sym(\supp(w_1))}.$$ Similarly
$$\left(N(\sigma_2) \cap N(\tau_2)\right) \subseteq
N(\sigma)\cap N(\tau)|_{\sym(\supp(w_2))}.$$ 
This implies that $E(w) \geq E(w_1) +
E(w_2) \geq 4(k-1)^2$, which proves the proposition. \end{proof}

\begin{lemma} \label{coxel} Let $W$ be a finite Coxeter group, and
$X$ be the conjugacy class of Coxeter elements of $W$. Then there
exists $w \in X$ of minimal length in $X$ such that $e(w) = E(w) =
0$.\end{lemma}
\begin{proof} Since the Coxeter graph of every finite Coxeter
group is a forest (that is, a disjoint union of trees), the
Coxeter graph of $W$ is 2-colourable. Therefore we may divide the
fundamental reflections into two sets $R = R_1 \dot\cup R_2$, such
that for $r, s \in R_i$, $rs = sr$. Now let $x = \prod_{r \in R_1}
r$ and $y = \prod_{r \in R_2} r$. Then $w = xy$ is an element of
$X$ of minimal length, and $e(w) = E(w) = 0$.\end{proof}\\

We remark that if $W$ is irreducible of rank at least two, its Coxeter graph can only
be 2-coloured in one way and consequently there are exactly two
Coxeter elements for which $e(w) = E(w) = 0$.\\

Our next proposition looks at cuspidal classes. The cuspidal classes in $W(B_n)$ are Coxeter elements of a reflection subgroup that is a direct product of Coxeter groups of type $B$. Therefore a cuspidal element is conjugate to a product of Coxeter elements of these factors. The construction of the elements $\tau_i$ and $\sigma_i$ in the proof of Proposition \ref{abdmin} may be viewed as parallel to that of $x$ and $y$ in the proof of Lemma \ref{coxel}. (We remark that caution must be exercised, as positive roots in the root system of a reflection subgroup are not necessarily positive roots in the root system $\Phi$ of $W$.)

\begin{prop} \label{abdmin} Let $W$ be irreducible of type $A_{n}$, $B_n$ or
$D_n$ and $X$ be a cuspidal conjugacy class of $W$. Then there exists $w \in X$ such that $e(w) =
E(w) = 0$ and $w$ has minimal length in $X$. \end{prop}

\begin{proof} 
Suppose $W = W(A_{n})$. The only cuspidal class is the class of
Coxeter elements, and so by Lemma \ref{coxel} the proposition holds. \\

It remains to deal with the case where $W$ is
type $B_n$ or type $D_n$ and $X$ is cuspidal. Here, by Proposition
\ref{conjbd}, the elements of $X$ contain only negative cycles of
lengths $\lambda_1 \geq  \ldots \geq \lambda_k$, where $\sum
\lambda_i = n$. We will construct an element $w \in X$ of minimal
length. Let $\nu_1 = 0$ and for $2 \leq i \leq k$ let $\nu_i =
\sum_{j=1}^{i-1} \lambda_i$. As in Proposition \ref{conjbd}, for $
1\leq i \leq k$, let $\mu_i = \sum_{j = i+1}^k \lambda_i$. Note
that $n= \mu_i + \lambda_i + \nu_i$. By Proposition \ref{conjbd},
the minimal length of elements of $X$, $l_{\min}(X)$, is given by
$$ l_{\min}(X) = \left\{\begin{array}{ll} \sum_{i=1}^k \left(
\lambda_i +
2\mu_i\right) & {\mbox{ if $W = W(B_n)$;}}\\
\sum_{i=1}^k \left( \lambda_i -1 + 2\mu_i\right) & {\mbox{ if $W = W(D_n)$.}}\end{array}
\right.$$ Now for $1 \leq i \leq k$ define two involutions, $\tau_i$ and $\sigma_i$ as
follows.
\begin{eqnarray*} \tau_i &=& \left\{\begin{array}{ll}
(\overset{+}{\nu_i + 1} \;\; \overset{+}{\nu_i +
2})(\overset{+}{\nu_i + 3} \;\; \overset{+}{\nu_i + 4})\cdots
(\overset{+}{\nu_{i} + \lambda_i - 2} \;\; \overset{+}{\nu_i +
\lambda_i - 1})(\overset{-}{\nu_i + \lambda_i}) & {\mbox{ if
$\lambda_i$ is odd}}\\
(\overset{+}{\nu_i + 2} \;\; \overset{+}{\nu_i +
3})(\overset{+}{\nu_i + 4} \;\; \overset{+}{\nu_i + 5})\cdots
(\overset{+}{\nu_{i} + \lambda_i - 2} \;\; \overset{+}{\nu_i +
\lambda_i - 1})(\overset{-}{\nu_i + \lambda_i}) & {\mbox{ if $\lambda_i$ is even}} \end{array}\right.\\
\sigma_i &=& \left\{\begin{array}{ll} (\overset{+}{\nu_i + 2} \;\;
\overset{+}{\nu_i + 3})(\overset{+}{\nu_i + 4} \;\;
\overset{+}{\nu_i + 5})\cdots (\overset{+}{\nu_{i} + \lambda_i -
1} \;\; \overset{+}{\nu_i +
\lambda_i}) & {\mbox{ if $\lambda_i$ is odd}}\\
(\overset{+}{\nu_i + 1} \;\; \overset{+}{\nu_i +
2})(\overset{+}{\nu_i + 3} \;\; \overset{+}{\nu_i + 4})\cdots
(\overset{+}{\nu_{i} + \lambda_i - 1} \;\; \overset{+}{\nu_i +
\lambda_i}) & {\mbox{ if $\lambda_i$ is even}} \end{array}\right.\\
\end{eqnarray*}
Define $w_i = \tau_i\sigma_i$, $\tau = \tau_1\cdots \tau_k$, $\sigma = \sigma_1 \cdots
\sigma_k$ and $w = \tau\sigma = w_1\cdots w_k$. Note that if $w \in W(D_n)$, then $k$ is
even and hence $\tau$ and $\sigma$ are both elements of $W(D_n)$. If $\lambda_i$ is odd,
then
$$w_i = (\overset{+}{\nu_i + 1} \;\; \overset{+}{\nu_i + 2} \;\;
\overset{+}{\nu_i + 4} \;\;  \cdots \overset{+}{\nu_i + \lambda_i
-3} \;\; \overset{-}{\nu_i + \lambda_i-1} \;\; \overset{+}{\nu_i +
\lambda_i} \;\; \overset{+}{\nu_i + \lambda_i-2} \cdots
\overset{+}{\nu_i + 5} \;\; \overset{+}{\nu_i + 3})$$ whereas if
$\lambda_i$ is even, then
$$w_i = (\overset{+}{\nu_i + 1} \;\; \overset{+}{\nu_i + 3} \;\; \overset{+}{\nu_i + 5}
\cdots \overset{+}{\nu_i + \lambda_i-3} \;\; \overset{-}{\nu_i +
\lambda_i-1} \;\; \overset{+}{\nu_i + \lambda_i} \;\;
\overset{+}{\nu_i + \lambda_i -2} \cdots \overset{+}{\nu_i + 4}
\;\; \overset{+}{\nu_i + 2}).
$$
In either case, $w_i$ is a negative $\lambda_i$-cycle. Hence $w
\in X$. We now consider $\ell(\tau)$ and $\ell(\sigma)$. Firstly,
note that $\sigma$ is a product of distinct mutually commuting
fundamental reflections. Hence $\ell(\sigma) = \sum_{i=1}^k
\lceil\frac{\lambda_i-1}{2}\rceil$. Now $\tau$ can be split into a
product $\tau = \tau'\tau''$ where $\tau'$ consists of the
2-cycles of $\tau$ and $\tau'' =
(\overset{-}{\nu_1})(\overset{-}{\nu_2})\cdots
(\overset{-}{\nu_k})$. Now $\ell(\tau') = \sum_{i=1}^k
\lfloor\frac{\lambda_i-1}{2}\rfloor$ and $N(\tau)$ consists of
roots of the form $e_a - e_b$ where for some $i$, $\{a, b\}
\subseteq \{\nu_i + 1, \ldots, \nu_i + \lambda_i\}$. Also it is
not hard to see that, thinking of $\tau''$ as an element of $W(B_n)$,%
\[ N(\tau'') = \cup_{i=1}^k \left[ \{e_{\nu_i+\lambda_i}\} \cup \{e_{\nu_i + \lambda_i}
\pm e_{b}: \nu_i + \lambda_i < b \leq n\}\right].\] %
Therefore $N(\tau') \cap N(\tau'') = \emptyset$. Hence $\ell(\tau) = \ell(\tau') +
\ell(\tau'')$. If $W = W(B_n)$ we get \linebreak $\ell(\tau'') = k + 2\sum_{i=1}^k \mu_i$ and hence
$\ell(\tau) =\sum_{i=1}^k \lfloor\frac{\lambda_i-1}{2}\rfloor + k + 2\sum_{i=1}^k \mu_i$.
If $W = W(D_n)$ then $\ell(\tau'') =  2\sum_{i=1}^k \mu_i$ and hence $\ell(\tau)
=\sum_{i=1}^k \lfloor\frac{\lambda_i-1}{2}\rfloor +
2\sum_{i=1}^k \mu_i$. \\

Now $w = \tau\sigma$, so $\ell(w) \leq \ell(\tau) + \ell(\sigma)$.
If $W =
W(B_n)$, then \begin{eqnarray*} \ell(w) &\leq& \ell(\tau) + \ell(\sigma) \\
&=& \sum_{i=1}^k \lfloor\frac{\lambda_i-1}{2}\rfloor + k +
2\sum_{i=1}^k \mu_i + \sum_{i=1}^k
\lceil\frac{\lambda_i-1}{2}\rceil \\
&=& \sum_{i=1}^k \lambda_i + 2\mu_i\\
&=& l_{\min}(X) \leq \ell(w).\end{eqnarray*}%
A similar calculation when $W = W(D_n)$ shows again that $\ell(w) \leq \ell(\tau) +
\ell(\sigma) = l_{\min}(X) \leq \ell(w)$. Therefore in each case, $\ell(w) = \ell(\tau) +
\ell(\sigma) = l_{\min}(X)$. Hence $e(w) = 0$ and $w$ has minimal length in $X$. Since
$w$ is cuspidal, $E(w) = e(w) = 0$ and we have therefore proved Proposition \ref{abdmin}.
\end{proof}

\paragraph{Proof of Theorem \ref{ex0minlength}} 
We argue by induction on the rank of $W$ -- when $W$ has rank one the theorem clearly holds. By induction and
 Lemma \ref{basics}(iii) we may suppose
$W$ is an irreducible Coxeter group. If $X \cap W_J \neq \emptyset$ for some $J {\small{\varsubsetneqq}} R$, then induction and Lemma \ref{minpara} yield the result. So we may also assume $X$ is a cuspidal class of $W$. Then the cases when  $W$ is of type $A_{n}, B_n$ and $D_n$ are dealt with in Proposition \ref{abdmin}, while the case when $W$ is isomorphic to
$\dih(2m)$ is easy to handle. For the remaining exceptional irreducible Coxeter groups we
had to resort to using {\sc Magma} \cite{magma}. The sizes of $X_{\min}^0$, where
$X_{\min}^0$ denotes the set of elements in X which have minimal length in $X$ and excess
zero, are itemized in the appendix as well as giving a representative element of this
set. In all cases $X_{\min}^0$ is non-empty, whence Theorem \ref{ex0minlength} is proven.

\qed

\section*{Appendix}

Tables 2--5 give cuspidal classes in exceptional Weyl groups. Each class is labeled using
the system in Carter \cite{carter}. Detailed information about these classes is given in
\cite{gkpf}, Tables B.1 -- B.6.  For a conjugacy class $X$ let $X_{\min}$ be the set of
elements of minimal length in $X$ and $X^0_{\min}$ be the set of $w \in X_{\min}$ for
which $e(x) = E(x) = 0$. For each class $X$ we give $|X_{\min}|$, $|X^0_{\min}|$ and a
representative $w \in X^0_{\min}$. Tables 6 and 7 give cuspidal classes in $H_3$ and
$H_4$. For these tables the label of the class is its address in the CHEVIE list
\cite{chevie}. See also \cite{gkpf}. We have included Coxeter elements for completeness, even though they are covered by Lemma \ref{coxel}. And finally, to
improve readability, we write the element, say,
$r_1r_3r_2r_4$ as $1324$.

\begin{table}[h!bt]
\begin{center}
\begin{tabular}{llll}
 \hline \vspace*{-4mm}\\ Class $X$ & $|X_{\min}|$ & $|X^0_{\min}|$ &
$w \in X^0_{\min}$
\\ \hline
$F_4$ &8&2&1324\\
$B_4$ &14&2&124323\\
$F_4(a_1)$&16&16&12132343\\
$D_4$&8&6&1232343234\\
$C_3 + A_1$&8&6&1213213234\\
$D_4(a_1)$&12&12&121321343234\\
$A_3 + \tilde A_1$&16&12&12132132343234\\
$A_2 + \tilde A_2$&16&16&1213213432132343\\
$4A_1$&1&1&$w_0$\\
 \hline
\end{tabular}
\vspace*{-3mm}\caption{Cuspidal Classes in $F_4$} \label{table3}
\end{center}
\end{table}

\afterpage\clearpage

\begin{table}[h!bt]
\begin{center}
\begin{tabular}{llll}
 \hline \vspace*{-4mm}\\ Class $X$ & $|X_{\min}|$ & $|X^0_{\min}|$ &
$w \in X^0_{\min}$
\\ \hline
$E_6$&32&2&142365\\
$E_6(a_1)$&80&4&12542346\\
$E_6(a_2)$&144&36&231423154654\\
$A_5+A_1$&48&10&23423454231465\\
$3A_2$&80&80&123142314542314565423456\\
 \hline
\end{tabular}
\vspace*{-3mm}\caption{Cuspidal Classes in $E_6$} \label{table4}
\end{center}
\end{table}

\begin{table}[h!bt]
\begin{center}
\begin{tabular}{llll}
 \hline \vspace*{-4mm}\\ Class $X$ & $|X_{\min}|$ & $|X^0_{\min}|$ &
$w \in X^0_{\min}$
\\ \hline
$E_7$&64&2&1423657\\
$E_7(a_1)$&160&4&125423476\\
$E_7(a_2)$&280&18&12465423457\\
$E_7(a_3)$&366&40&1231542365476\\
$D_6 + A_1$&96&10&123423465423457\\
$A_7$&316&58&12365423476542345\\
$E_7(a_4)$&800&422&234354231435465765431\\
$D_6(a_2)+ A_1$&708&190&13542345654231435765423 \\
$A_5+A_2$&420&194&1234231436542314576542345\\
$D_4 + 3A_1$&32&20&1234234542345654234567654234567\\
$2A_3 + A_1$&360&326&423454365423143542654317654234567 \\
$7A_1$&1&1&$w_0$\\
 \hline
\end{tabular}
\vspace*{-3mm}\caption{Cuspidal Classes in $E_7$} \label{table5}
\end{center}
\end{table}

\pagebreak

\begin{table}[h!bt]
\begin{center}
\begin{tabular}{llll}
 \hline \vspace*{-4mm}\\ Class $X$ & $|X_{\min}|$ & $|X^0_{\min}|$ &
$w \in X^0_{\min}$
\\ \hline
$E_8$&128&2&14682357\\
$E_8(a_1)$&320&4&1425423768\\
$E_8(a_2)$&624&10&435423145768\\
$E_8(a_4)$&732&40&12316542376548\\
$E_8(a_5)$&1516&66&1231654237687654\\
$E_7+A_1$&192&10&2423454231437658\\
$D_8$&852&18&242345423165438765\\
$E_8(a_3)$&2696&238&23423546542314768765\\
$D_8(a_1)$&2040&178&1245423654765423876543\\
$E_8(a_7)$&2360&316&3542314654231765423148\\
$E_8(a_6)$&3370&422&235423143546765423145876\\
$E_7(a_2)+A_1$&1758&100&134234542314354265423876\\
$E_6 + A_2$&840&194&12345423143546542314354876\\
$D_8(a_2)$&4996&362&31423547654234587654231456\\
$A_8$&2816&592&1234231436542314765876542345\\
$D_8(a_3)$&7748&910&231423145423657654231435487654\\
$D_6+2A_1$&256&48&34354231435426542345687654234567\\
$A_7+A_1$&2080&422&2314254236542345676542314354265478\\
$E_8(a_8)$&4480&4480&1231431542345654231456765423143546787654\\
$E_7(a_4) + A_1$&11592&5158&123435423456542317654231435426543765876543\\
$2D_4$&4070&1262&12425423456542765423456787654231435426543768\\
$E_6(a_2)+A_2$&16374&7438&34231454234654231435465765428765423143542657\\
$A_5+A_2+A_1$&3752&1910&3142345423654234567654234567876542314354265478\\
$D_5(a_1)+A_3$&15134&4900&4254231435465423143542678765423143542654317687\\
$2A_4$&7952&4058&134542314354265423143546765423143567876542345678\\
$2D_4(a_1)$&15120&15120&142354265423176542314354265487654231435426543176$\cdot$ \\ &&& $\cdot$542348765423\\
$D_4+4A_1$&56&42&124231542314365423143542765423143542654318765423$\cdot$ \\ &&& $\cdot$1435426543176543\\
$2A_3+2A_1$&1260&1090&123142315431654765423143542654387654231435426543$\cdot$
\\ &&& $\cdot$176542345876542345\\
$4A_2$&4480&4480&123142314542314354265437654231435426543176876542$\cdot$ \\ &&& $\cdot$31435426543176542345678765423456\\
$8A_1$&1&1&$w_0$\\
 \hline
\end{tabular}
\vspace*{-3mm}\caption{Cuspidal Classes in $E_8$} \label{table6}
\end{center}
\end{table}

\begin{table}[h!bt]
\begin{center}
\begin{tabular}{llll}
 \hline \vspace*{-4mm}\\ Class $X$ & $|X_{\min}|$ & $|X^0_{\min}|$ &
$w \in X^0_{\min}$
\\ \hline
6&4&2&312\\
8&6&4&21231\\
9&6&6&121232123\\
10&1&1&$w_0$\\
 \hline
\end{tabular}
\vspace*{-3mm}\caption{Cuspidal Classes in $H_3$} \label{table7}
\end{center}
\end{table}

\begin{table}[h!bt]
\begin{center}
\begin{tabular}{llll}
 \hline \vspace*{-4mm}\\ Class $X$ & $|X_{\min}|$ & $|X^0_{\min}|$ &
$w \in X^0_{\min}$
\\ \hline
11&8&2&1324\\
14&12&4&212431\\
15&18&8&32434121\\
16&22&10&2321234121\\
18&24&24&121213212343\\
19&34&26&32121321234121\\
21&12&10&2132121321234121\\
22&24&20&2321234321234121\\
23&38&30&213212134321234121\\
24&40&40&12121321213214321234\\
25&36&28&2132121321234321234121\\
26&24&24&121213212134321213212343\\
27&56&52&12123212132123432121321234\\
28&40&38&1212132123432121321234321234\\
29&60&60&121213212132124321213214321234\\
30&24&24&121213212132432121321234321213212343\\
31&36&34&12121321213212343212132123432121321234\\
32&40&40&1212132121321243212132124321213214321234\\
33&24&24&1212132121321432121321234321213212343212$\cdot$ \\ &&& $\cdot$13212343\\
34&1&&$w_0$\\
 \hline
\end{tabular}
\vspace*{-3mm}\caption{Cuspidal Classes in $H_4$} \label{table8}
\end{center}
\end{table}

\end{document}